\documentclass[12pt]{amsart}
\usepackage{amscd, amsfonts, amssymb, amsthm, mathrsfs}
\usepackage{parskip, fullpage, verbatim}
\usepackage[all]{xypic}
\usepackage{ctable}
\usepackage{longtable}
\usepackage{array}
\usepackage{aliascnt}
\usepackage[colorlinks, breaklinks, linkcolor=blue]{hyperref} 
\usepackage{breakurl}
\usepackage{booktabs}
\usepackage{wasysym}

\newcommand\CA{{\mathscr A}}

\newcommand\CIF{{\mathcal {IF}}} 
\newcommand\HIF{{\mathcal {HIF}}}
 
\newcommand\CDF{{\mathcal {DF}}} 
\newcommand\HDF{{\mathcal {HDF}}}

\newcommand\BBC{{\mathbb C}}

\newcommand\BBK{{\mathbb K}}

\newcommand\BBZ{{\mathbb Z}}

\newcommand\GL{\operatorname{GL}}

\newcommand\Der{\operatorname{Der}}
\newcommand\pdeg{\operatorname{pdeg}}

\renewcommand\th{{^{\text{th}}}}

\numberwithin{equation}{section}

\theoremstyle{plain}

\newtheorem{theorem}[equation]{Theorem}
\newtheorem{conjecture}[equation]{Conjecture}
\newtheorem{corollary}[equation]{Corollary}
\newtheorem{proposition}[equation]{Proposition}
\theoremstyle{definition}
\newtheorem{definition}[equation]{Definition}
\newtheorem{remark}[equation]{Remark}
\newtheorem{example}[equation]{Example}

\begin{document}


\title[Divisionally free Restrictions of Reflection Arrangements]{Divisionally free Restrictions of Reflection Arrangements}


\author[G. R\"ohrle]{Gerhard R\"ohrle}
\address
{Fakult\"at f\"ur Mathematik,
Ruhr-Universit\"at Bochum,
D-44780 Bochum, Germany}
\email{gerhard.roehrle@rub.de}


\begin{abstract}
We study some aspects of
divisionally free arrangements 
which were recently introduced by Abe.
Crucially, 
Terao's conjecture on the combinatorial 
nature of freeness holds within this class.
We show that while it is compatible with 
products, surprisingly, it is
not closed under taking localizations.
In addition, we determine all
divisionally free restrictions 
of all reflection arrangements.
\end{abstract}

\maketitle
\allowdisplaybreaks


\section{Introduction}
The interplay between algebraic and combinatorial structures of 
hyperplane arrangements has been a driving force in the study of the field 
for many decades. At the very heart of these investigations lies 
Terao's fundamental Conjecture \ref{conj:terao} which asserts that the 
algebraic property of freeness of an arrangement 
is determined by purely combinatorial data.

\begin{conjecture}
[{\cite[Conj.~4.138]{orlikterao:arrangements}}]
\label{conj:terao}
For a fixed field, 
freeness of the arrangement $\CA$ only depends on its 
lattice  $L(\CA)$, i.e.\ is combinatorial.
\end{conjecture}

The conjecture is known to be sensitive to 
a change of the 
characteristic of the underlying field,  
cf.~\cite[\S 4]{ziegler:matroid}. 
It is still open even for dimension $3$.

Recently, T.~Abe \cite{abe:divfree} introduced 
a new class of free hyperplane arrangements, so called
\emph{divisionally free} arrangements $\CDF$
(Definition \ref{def:divfree}). This 
properly encompasses the 
class of inductively free arrangements $\CIF$ (Definition \ref{def:indfree}),
cf.~\cite[Thm.~1.6]{abe:divfree}.
The relevance of this new notion is that 
Conjecture \ref{conj:terao} 
is valid within $\CDF$, cf.~\cite[Thm.~4.4(3)]{abe:divfree}.

Each of the classes of free, inductively free and recursively free
arrangements is compatible with 
the product construction for arrangements
(\cite[Prop.\ 4.28]{orlikterao:arrangements}, 
\cite[Prop.~2.10]{hogeroehrle:inductivelyfree}, 
\cite[Thm.~2]{hogeroehrleschauenburg:localizations}).
We show that this also is the case for $\CDF$, 
cf.~Proposition \ref{prop:product-divfree}.

All the previously mentioned classes of free arrangements are known to be 
closed with respect to taking localizations
(\cite[Thm.~4.37]{orlikterao:arrangements}, 
\cite[Thm.~1]{hogeroehrleschauenburg:localizations}).
Unexpectedly, this fails for Abe's new class 
$\CDF$, see Example \ref{ex:divlocalization}.

Because of its relevance to Conjecture \ref{conj:terao}, 
it is important to know which arrangements from a given class
belong to $\CDF$, e.g.\ see \cite[\S 6]{abe:divfree}. 
Reflection arrangements have had  
a pivotal role in the theory of hyperplane arrangements ever since.  
In \cite[Cor.~4.7]{abe:divfree}, Abe determined all irreducible 
divisionally free reflection arrangements, see 
Theorem \ref{thm:divfreereflections}.
We extend this classification to all restrictions of reflection arrangements
in Theorem \ref{thm:divfreerestrictions}.

For general information about arrangements and reflection groups 
we refer the reader to \cite{orliksolomon:unitaryreflectiongroups} and
\cite{orlikterao:arrangements}. In this article we use the classification 
and labeling of the irreducible unitary reflection groups
due to Shephard and Todd, \cite{shephardtodd}.

\section{Preliminaries}
\label{ssect:recoll}

\subsection{Hyperplane Arrangements}
\label{ssect:hyper}

Let $\BBK$ be a field and let $V = \BBK^\ell$.
By a hyperplane arrangement in $V$ we mean a finite set $\CA$ of
hyperplanes in $V$. Such an arrangement is denoted $(\CA,V)$ 
or simply $\CA$. If $\dim V = \ell$
we call $\CA$ an $\ell$-arrangement.
The number of elements in $\CA$ is given by $|\CA|$.
The empty $\ell$-arrangement 
is denoted by $\Phi_\ell$.

By $L(\CA)$ we denote the set of all nonempty intersections of elements of $\CA$,
\cite[Def.~1.12]{orlikterao:arrangements}.
For $X \in L(\CA)$, we have two associated arrangements, 
firstly the subarrangement 
$\CA_X :=\{H \in \CA \mid X \subseteq H\} \subseteq \CA$
of $\CA$ and secondly, 
the \emph{restriction of $\CA$ to $X$}, $(\CA^X,X)$, where 
$\CA^X := \{ X \cap H \mid H \in \CA \setminus \CA_X\}$,
\cite[Def.~1.13]{orlikterao:arrangements}.
Note that $V$ belongs to $L(\CA)$
as the intersection of the empty 
collection of hyperplanes and $\CA^V = \CA$. 

If $0 \in H$ for each $H$ in $\CA$, then 
$\CA$ is called \emph{central}.
We only consider central arrangements.

Let $H \in \CA$ (for $\CA \neq \Phi_\ell$) and
define $\CA' := \CA \setminus\{ H\}$,
and $\CA'' := \CA^{H}$.
Then $(\CA, \CA', \CA'')$ is a \emph{triple} of arrangements, 
\cite[Def.~1.14]{orlikterao:arrangements}.

The \emph{product}
$\CA = (\CA_1 \times \CA_2, V_1 \oplus V_2)$ 
of two arrangements $(\CA_1, V_1), (\CA_2, V_2)$
is defined by
\begin{equation*}
\label{eq:product}
\CA := \CA_1 \times \CA_2 = \{H_1 \oplus V_2 \mid H_1 \in \CA_1\} \cup 
\{V_1 \oplus H_2 \mid H_2 \in \CA_2\},
\end{equation*}
see \cite[Def.~2.13]{orlikterao:arrangements}.
In particular, $|\CA| = |\CA_1| + |\CA_2|$. 

An arrangement $\CA$ is called \emph{reducible},
if it is of the form $\CA = \CA_1 \times \CA_2$, where 
$\CA_i \ne \Phi_0$ for $i=1,2$, else $\CA$
is \emph{irreducible}, 
\cite[Def.~2.15]{orlikterao:arrangements}.

If $\CA = \CA_1 \times \CA_2$ is a product, 
then by \cite[Prop.~2.14]{orlikterao:arrangements}
there is a lattice isomorphism
\[
 L(\CA_1) \times L(\CA_2) \cong L(\CA) \quad \text{by} \quad
(X_1, X_2) \mapsto X_1 \oplus X_2.
\]
It is easy to see that
for $X =  X_1 \oplus X_2 \in L(\CA)$, we have 
\begin{equation}
\label{eq:restrproduct}
\CA^X = \CA_1^{X_1} \times \CA_2^{X_2}.
\end{equation}

The \emph{characteristic polynomial} 
$\chi(\CA,t) \in \BBZ[t]$ 
of $\CA$ is defined by 
\[
\chi(\CA,t) := \sum_{X \in L(\CA)} \mu(X)t^{\dim X},
\]
where $\mu$ is the M\"obius function of $L(\CA)$, 
see \cite[Def.\ 2.52]{orlikterao:arrangements}.

If $\CA = \CA_1 \times \CA_2$ is a product, then, 
thanks to \cite[Lem.\ 2.50]{orlikterao:arrangements}, 
\begin{equation}
\label{eq:chiproduct}
\chi(\CA,t) = \chi(\CA_1,t) \cdot  \chi(\CA_2,t). 
\end{equation}

\subsection{Free Arrangements}
\label{ssect:free}
Let $S = S(V^*)$ be the symmetric algebra of the dual space $V^*$ of $V$.
If $\CA$ is an arrangement in $V$,
then for every $H \in \CA$ we may fix $\alpha_H \in V^*$ with
$H = \ker(\alpha_H)$.
We call $Q(\CA) := \prod_{H \in \CA} \alpha_H \in S$
the \emph{defining polynomial} of $\CA$.

The \emph{module of $\CA$-derivations} is the $S$-submodule of $\Der(S)$,
the $S$-module of $\BBK$-derivations of $S$, 
defined by 
\[
D(\CA) := \{\theta \in \Der(S) \mid \theta(Q(\CA)) \in Q(\CA) S\}.
\]
The arrangement $\CA$ is said to be \emph{free} 
if $D(\CA)$ is a free $S$-module.

If $\CA$ is a free $\ell$-arrangement, 
then $D(\CA)$ admits an $S$-basis of $\ell$ 
homogeneous derivations $\theta_1, \ldots, \theta_\ell$, 
by \cite[Prop.~4.18]{orlikterao:arrangements}.
The \emph{exponents} of the free arrangement $\CA$ are given by the multiset
given by the polynomial degrees of the $\theta_i$, 
$\exp\CA := \{\pdeg \theta_1, \ldots, \pdeg \theta_\ell\}$.

Terao's basic \emph{Addition-Deletion Theorem} 
plays a key role in the study of free arrangements.

\begin{theorem}
[{\cite{terao:freeI}, \cite[Thm.\ 4.51]{orlikterao:arrangements}}]
\label{thm:add-del}
Suppose $\CA \neq \Phi_\ell$ and
let $(\CA, \CA', \CA'')$ be a triple of arrangements. Then any 
two of the following statements imply the third:
\begin{itemize}
\item[(i)] $\CA$ is free with $\exp\CA = \{ b_1, \ldots , b_{\ell -1}, b_\ell\}$;
\item[(ii)] $\CA'$ is free with $\exp\CA' = \{ b_1, \ldots , b_{\ell -1}, b_\ell-1\}$;
\item[(iii)] $\CA''$ is free with $\exp\CA'' = \{ b_1, \ldots , b_{\ell -1}\}$.
\end{itemize}
\end{theorem}

The following is Terao's celebrated  
\emph{Factorization Theorem} for  
free arrangements.

\begin{theorem}
[{\cite[Thm.\ 4.137]{orlikterao:arrangements}}]
\label{thm:factorization}
If $\CA$ is free with $\exp\CA = \{ b_1, \ldots , b_{\ell -1}, b_\ell\}$, then 
\[
\chi(\CA,t) = \prod\limits_{i=1}^\ell (t - b_i).
\]
\end{theorem}

\subsection{Inductively Free Arrangements}
\label{ssect:indfree}

An iterative application of the 
addition part of Theorem \ref{thm:add-del} leads to the class of  
\emph{inductively free} arrangements.

\begin{definition}
[{\cite[Def.~4.53]{orlikterao:arrangements}}]
\label{def:indfree}
The class $\CIF$ of \emph{inductively free} arrangements 
is the smallest class of arrangements subject to
\begin{itemize}
\item[(i)] $\Phi_\ell \in \CIF$ for each $\ell \ge 0$;
\item[(ii)] if there exists a hyperplane $H_0 \in \CA$ such that both
$\CA'$ and $\CA''$ belong to $\CIF$, and $\exp \CA '' \subseteq \exp \CA'$, 
then $\CA$ also belongs to $\CIF$.
\end{itemize}
\end{definition}

There is a hereditary version of 
$\CIF$, cf.~\cite[\S 6.4]{orlikterao:arrangements}.

\begin{definition}
\label{def:heredindfree}
The arrangement $\CA$ is called 
\emph{hereditarily inductively free} provided 
that $\CA^X$ is inductively free for each $X \in L(\CA)$.
Denote this class by $\HIF$.
\end{definition}

Note that if $\CA$ is hereditarily inductively free, 
it is inductively free as $V \in L(\CA)$ and $\CA^V = \CA$.

\subsection{Reflection Arrangements}
\label{ssect:refl}

Let $W \subseteq \GL(V)$ be a finite, 
complex reflection group acting on the complex vector space $V=\BBC^\ell$.
The \emph{reflection arrangement} of $W$ in $V$ is the 
hyperplane arrangement $\CA(W)$ 
consisting of the reflecting hyperplanes 
of the elements in $W$ acting as reflections on $V$.

Terao \cite{terao:freeI} has shown that every 
reflection arrangement $\CA(W)$ is free 
and that the exponents of $\CA(W)$
coincide with the coexponents of $W$, 
cf.~\cite[Prop.~6.59 and Thm.~6.60]{orlikterao:arrangements}. 

We recall the classification of the 
inductively free reflection arrangements from 
\cite{hogeroehrle:inductivelyfree}.

\begin{theorem}
[{\cite[Thms.~1.1 and 1.2]{hogeroehrle:inductivelyfree}}]
\label{thm:indfree1}
Let $W$ be a finite, irreducible, complex 
reflection group with reflection arrangement 
$\CA(W)$. Then the following hold:
\begin{itemize}
\item[(i)] $\CA(W)$ is 
inductively free if and only if 
$W$ does not admit an irreducible factor
isomorphic to a monomial group 
$G(r,r,\ell)$ for $r, \ell \ge 3$, 
$G_{24}, G_{27}, G_{29}, G_{31}, G_{33}$, or $G_{34}$.
\item[(ii)] $\CA(W)$ is inductively free
if and only if it is hereditarily inductively free.
\end{itemize}
\end{theorem}

\subsection{Restrictions of Reflection Arrangements}
\label{ssec:akl}

Orlik and Solomon defined intermediate 
arrangements $\CA^k_\ell(r)$ in 
\cite[\S 2]{orliksolomon:unitaryreflectiongroups}
(cf.\ \cite[\S 6.4]{orlikterao:arrangements}) which
interpolate between the
reflection arrangements of $G(r,r,\ell)$ and $G(r,1,\ell)$. 
They show up as restrictions of the reflection arrangement
of $G(r,r,\ell)$, 
\cite[Prop.\ 2.14]{orliksolomon:unitaryreflectiongroups} 
(cf.~\cite[Prop.\ 6.84]{orlikterao:arrangements}).

For 
$\ell \geq 2$ and $0 \leq k \leq \ell$ the defining polynomial of
$\CA^k_\ell(r)$ is given by
\[
Q(\CA^k_\ell(r)) = x_1 \cdots x_k\prod\limits_{\substack{1 \leq i < j \leq \ell\\ 0 \leq n < r}}(x_i - \zeta^nx_j),
\]
where $\zeta$ is a primitive $r\th$ root of unity,
so that 
$\CA^\ell_\ell(r) = \CA(G(r,1,\ell))$ and 
$\CA^0_\ell(r) = \CA(G(r,r,\ell))$. 
For $k \neq 0, \ell$, these are not reflection arrangements
themselves. 

We recall the classification of the 
inductively free restrictions of all  
reflection arrangements. 

\begin{theorem}
[{\cite[Thms.~1.2 and 1.3]{amendhogeroehrle:indfree}}]
\label{thm:indfree}
Let $W$ be a finite, irreducible, complex 
reflection group with reflection arrangement 
$\CA(W)$ and let $V \ne X \in L(\CA(W))$. 
\begin{itemize}
\item[(a)] 
$\CA(W)^X$ is inductively free 
if and only if one of the following holds:
\begin{itemize}
\item[(i)] 
$\CA$ is inductively free;
\item[(ii)] 
$W = G(r,r,\ell)$, $r \ge 3$ and 
$\CA(W)^X \cong \CA^k_p(r)$, for $p = \dim X$ and $p - 2 \leq k \leq p$; 
\item[(iii)] 
$W = G_{24}, G_{27}, G_{29}, G_{31}, G_{33}$, or $G_{34}$ 
and $X \in L(\CA(W))$ with  $\dim X \leq 3$.
\end{itemize}
\item[(b)]
$\CA(W)^X$ is inductively free 
if and only if it is hereditarily inductively free.
\end{itemize}
\end{theorem}

\subsection{Divisionally Free Arrangements}
\label{ssect:divfree}
First we recall the key result from \cite{abe:divfree}.

\begin{theorem}
[{\cite[Thm.~1.1]{abe:divfree}}]
\label{thm:abe-div}
Let $\CA \ne \Phi_\ell$.
Suppose there is a hyperplane $H$ in $\CA$ such that the restriction 
$\CA^H$ is free and that 
$\chi(\CA^H,t)$ divides $\chi(\CA,t)$.
Then $\CA$ is free.
\end{theorem}
  
Theorem \ref{thm:abe-div} can be viewed as a strengthening of the addition
part of Theorem \ref{thm:add-del}. An iterative  
application leads to the class $\CDF$.

\begin{definition}
[{\cite[Def.~1.5]{abe:divfree}}]
\label{def:divfree}
An $\ell$-arrangement $\CA$ is called 
\emph{divisionally free} if $\ell \le 2$, $\CA = \Phi_\ell$, or
else there is a sequence of consecutive restrictions of arrangements
starting with $\CA$ and ending in a $2$-arrangement such that
the successive characteristic polynomials divide one another.
That is, there is a sequence of arrangements
$\CA = \CA_\ell, \CA_{\ell-1}, \ldots, \CA_2$
such that for each $i = 3, \ldots, \ell$
there is an $H_i$ in $\CA_i$ so that $\CA_i^{H_i} = \CA_{i-1}$
and $\chi(\CA_i^{H_i},t)$ divides $\chi(\CA_i,t)$.
Denote this class by $\CDF$.
\end{definition}

Thanks to Theorem \ref{thm:abe-div} 
and the fact that any $2$-arrangement is free, 
any $\CA$ in $\CDF$ is free.

In \cite[Thm.~1.6]{abe:divfree}, Abe observed that $\CIF \subsetneq \CDF$.
The reflection arrangement of the complex reflection group $G_{31}$ is 
divisionally free but not inductively free.

\begin{theorem}
[{\cite[Thm.~5.6]{abe:divfree}}]
\label{thm:akl-divfree}
Let $r, \ell \ge 3$. Then $\CA^k_\ell(r) \in \CDF$ if and only if $k \ne 0$.
\end{theorem}

It is immediate from Theorem \ref{thm:akl-divfree} and 
\cite[Prop.\ 2.11]{orliksolomon:unitaryreflectiongroups} 
(cf.~\cite[Prop.\ 6.82]{orlikterao:arrangements})
that the class  $\CDF$ is not closed under restrictions.

Each of the classes of free, inductively free and recursively free
arrangements is compatible with 
the product construction for arrangements, 
cf.~\cite[Prop.\ 4.28]{orlikterao:arrangements}, 
\cite[Prop.~2.10]{hogeroehrle:inductivelyfree}, 
\cite[Thm.~2]{hogeroehrleschauenburg:localizations}.
We observe that this also holds for the class $\CDF$.

\begin{proposition}
\label{prop:product-divfree}
Let $(\CA_1, V_1),  (\CA_2, V_2)$ be two arrangements.
Then  $\CA = (\CA_1 \times \CA_2, V_1 \oplus V_2)$ is 
divisionally free if and only if both 
$\CA_1$ and $\CA_2$ are 
divisionally free.
\end{proposition}

\begin{proof}
First suppose that both $\CA_1$ and $\CA_2$ are divisionally free.
Then we claim that also $\CA$ is so.
We argue via induction on $|\CA|$. If both $\CA_1$ and $\CA_2$ are empty, 
there is nothing to show.
So suppose that $|\CA| \ge 1$ and 
that the claim holds for any 
product of divisionally free arrangements with fewer   than $|\CA|$ 
hyperplanes.
Without loss of generality there exists an $H_1$ in $\CA_1$ 
such that $\CA_1^{H_1}$ belongs to $\CDF$
and $\chi( \CA_1^{H_1}, t)$ divides $\chi( \CA_1, t)$.
Letting $H = H_1 \oplus V_2 \in \CA$, by \eqref{eq:restrproduct}, 
$\CA^H = \CA_1^{H_1} \times \CA_2$ is a product of 
divisionally free arrangements with $|\CA^H| < |\CA|$. 
So, by our induction hypothesis, 
$\CA^H$ is divisionally free.
In addition, since 
$\chi( \CA^{H}, t) = \chi( \CA_1^{H_1}, t) \cdot \chi( \CA_2, t)$ divides 
$\chi( \CA_1, t) \cdot \chi( \CA_2, t) = \chi( \CA, t)$, 
cf.~\eqref{eq:chiproduct}, 
we infer that $\CA$ belongs to $\CDF$.

Conversely, suppose that 
$\CA = \CA_1 \times \CA_2$ belongs to 
$\CDF$. We claim that then both $\CA_1$ and $\CA_2$ also belong to $\CDF$.
Again we argue by induction on $|\CA|$.
If both $\CA_1$ and $\CA_2$ are empty, 
there is nothing to show.
So suppose that $|\CA| \ge 1$ 
and that the claim holds for any 
product in $\CDF$ with fewer than $|\CA|$ hyperplanes.
Without loss of generality we may assume that there is an 
$H = H_1 \oplus V_2$ in $\CA$ such that $\CA^H$ belongs to $\CDF$
and $\chi( \CA^H, t)$ divides $\chi( \CA, t)$.
Since $|\CA^H| < |\CA|$ and  
$\CA^H = \CA_1^{H_1} \times \CA_2$, 
both $\CA_1^{H_1}$ and $\CA_2$ belong to $\CDF$, by our 
induction hypothesis.
Moreover,  since 
$\chi( \CA^{H}, t) = \chi( \CA_1^{H_1}, t) \cdot \chi( \CA_2, t)$ 
and 
$\chi( \CA, t) = \chi( \CA_1, t) \cdot \chi( \CA_2, t)$,
cf.~\eqref{eq:chiproduct}, 
it follows that 
$\chi( \CA_1^{H_1}, t)$ divides $\chi( \CA_1, t)$.
Therefore, also $\CA_1$ belongs to $\CDF$.
\end{proof}

There is a hereditary version of $\CDF$. 

\begin{definition}
[{\cite[Def.~5.7(2)]{abe:divfree}}]
\label{def:hereddivfree}
The arrangement $\CA$ is called 
\emph{hereditarily divisionally free} provided 
that $\CA^X$ is divisionally free for each $X \in L(\CA)$.
We denote this class by $\HDF$.
\end{definition}

\begin{remark}
\label{rem:hereddiv}
(i).  Clearly, since $\CIF \subseteq \CDF$, we have $\HIF \subseteq \HDF$.

(ii). Note that $\HDF \subsetneq \CDF$. For, in \cite[Ex.~2.16]{hogeroehrle:inductivelyfree},
we constructed an inductively free arrangement $\CA$ which admits a hyperplane $H$ such 
that $\CA^H$ is not free. In particular, $\CA \in \CDF \setminus \HDF$.
\end{remark}

The compatibility with 
products from Proposition \ref{prop:product-divfree}
descends to $\HDF$.
The proof follows readily from Proposition \ref{prop:product-divfree} and 
\eqref{eq:restrproduct}.

\begin{corollary}
\label{cor:product-hereddivfree}
Let $\CA_1,  \CA_2$ be two arrangements.
Then  $\CA = \CA_1 \times \CA_2$ is 
hereditarily divisionally free if and only if both 
$\CA_1$ and $\CA_2$ are 
hereditarily divisionally free.
\end{corollary}

While the classes of free, inductively free and recursively free arrangements
are closed under taking localizations,  
cf.~\cite[Thm.~4.37]{orlikterao:arrangements}, 
\cite[Thm.~1]{hogeroehrleschauenburg:localizations}, 
unexpectedly, 
this property fails for $\CDF$, as
the following example illustrates.

\begin{example}
\label{ex:divlocalization}
Let $\CA = \CA^1_\ell(r)$ for $r \ge 3$ and $\ell \ge 4$.
Let 
\[
X = \bigcap\limits_{\substack{2 \leq i < j \leq \ell\\ 0 \leq n < r}} \ker(x_i - \zeta^nx_j),
\]
where $\zeta$ is a primitive $r\th$ root of unity.
Then $\CA_X \cong \CA^0_{\ell-1}(r)$.
By Theorem \ref{thm:akl-divfree},
$\CA$ is divisionally free but $\CA_X$ is not.
\end{example}

\begin{remark}
\label{rem:triples}
If $\CA \ne \Phi_\ell$ is inductively free and $(\CA, \CA', \CA'')$ 
is a triple as in Definition \ref{def:indfree}(ii), then 
each member of 
$(\CA, \CA', \CA'')$ belongs to $\CIF$. 
In contrast, if 
$\CA$ and $\CA''$ are consecutive members in a sequence
as in Definition \ref{def:divfree}, then 
the deletion $\CA'$ 
in the corresponding triple $(\CA, \CA', \CA'')$ 
need not belong to $\CDF$.
For instance, let $\CA = \CA^1_\ell(r)$ for $r, \ell \ge 3$ 
and $H = \ker(x_1)$.
Then $\CA^H \cong \CA^{\ell-1}_{\ell-1}(r)$, 
by \cite[Prop.\ 2.11]{orliksolomon:unitaryreflectiongroups}. 
It follows from 
\cite[Prop.\ 2.13]{orliksolomon:unitaryreflectiongroups} 
(or else from \cite[Prop.~6.85]{orlikterao:arrangements}
and Theorem \ref{thm:factorization})
that $\chi(\CA^H,t)$ divides $\chi(\CA,t)$.
Consequently, thanks to Theorem \ref{thm:akl-divfree},
$\CA$ and $\CA''$ are successive terms in a sequence
as in Definition \ref{def:divfree}.
However, $\CA' = \CA^0_\ell(r)$ is not divisionally free, by 
Theorem \ref{thm:akl-divfree}.
\end{remark}

\section{Divisionally Free Reflection Arrangements}
\label{sec:proofs}

In view of Theorem \ref{thm:indfree1}(i)
and Proposition \ref{prop:product-divfree}, we can restate 
Abe's classification of the
divisionally free reflection arrangements \cite[Cor.~4.7]{abe:divfree}
as follows.

\begin{theorem}
\label{thm:divfreereflections}
For  $W$ a finite, complex reflection group,
its reflection arrangement $\CA(W)$ is
divisionally free if and only if 
$W$ does not admit an irreducible factor
isomorphic to a monomial group 
$G(r,r,\ell)$ for $r, \ell \ge 3$, 
$G_{24}, G_{27}, G_{29}, G_{33}$, or $G_{34}$.
\end{theorem}

The counterpart to Theorem \ref{thm:indfree1}(ii)
also holds in the setting of divisional freeness:

\begin{theorem}
\label{thm:divheredfreereflections}
The reflection arrangement of a finite, 
complex reflection group
is divisionally free if and only if 
it is hereditarily divisionally free.
\end{theorem}

\begin{proof}
The forward implication follows from 
Proposition \ref{prop:product-divfree},
Theorems \ref{thm:indfree1}(ii) and \ref{thm:divfreereflections}, 
and \cite[Prop.~5.8]{abe:divfree}.
The reverse implication is obvious.
\end{proof}

We now consider restrictions of reflection arrangements.
Thanks to Proposition \ref{prop:product-divfree} and 
\eqref{eq:restrproduct}, the question of divisional   
freeness of $\CA^X$ reduces readily to the case when
$\CA$ is irreducible, so we may assume that $W$ is irreducible.
In view of Theorem \ref{thm:indfree}(a), 
we can formulate our classification as follows:

\begin{theorem}
\label{thm:divfreerestrictions}
Let $W$ be a finite, irreducible, complex 
reflection group with reflection arrangement 
$\CA(W)$ and let $V \ne X \in L(\CA(W))$. 
Then $\CA(W)^X$ is divisionally free 
if and only if one of the following holds:
\begin{itemize}
\item[(i)] 
$\CA(W)^X$ is inductively free;
\item[(ii)] 
$W = G(r,r,\ell)$ and 
$\CA(W)^X \cong \CA^k_p(r)$, where $p = \dim X$ and $1 \leq k \leq p-3$; 
\item[(iii)] 
$W = G_{33}$ or $G_{34}$ and $\dim X = 4$.
\end{itemize}
\end{theorem}

\begin{proof}
If $\CA(W)^X$ is inductively free, the result follows, 
since $\CIF \subseteq \CDF$.

Now suppose that $\CA(W)^X$ is not inductively free.
For $W = G(r,r,\ell)$, 
every restriction is of the form 
$\CA(W)^X \cong \CA^k_p(r)$, where $p = \dim X$ and $1 \leq k \leq p$,
by \cite[Prop.\ 2.14]{orliksolomon:unitaryreflectiongroups}. 
Thus (ii) follows from Theorems
\ref{thm:indfree}(a)(ii) and \ref{thm:akl-divfree}.
 
For $W = G_{33}$ or $G_{34}$ the result 
follows from 
\cite[Thm.~1.2]{orliksolomon:unitaryreflectiongroups},
\cite[Tables 10, 11]{orliksolomon:unitaryreflectiongroups},
along with Theorem \ref{thm:indfree}(a)(iii).
\end{proof}

In view of 
Proposition \ref{prop:product-divfree} and 
Theorem \ref{thm:indfree}, 
we can restate Theorem \ref{thm:divfreerestrictions} as follows.

\begin{theorem}
\label{thm:divfreereflections2}
Let $W$ be a finite, irreducible, complex 
reflection group 
and let $V \ne X \in L(\CA(W))$. 
Then $\CA(W)^X$ is divisionally free 
unless $W = G_{34}$ and $\dim X = 5$.
\end{theorem}

Finally, we show that Theorem \ref{thm:divheredfreereflections}
extends to restrictions.

\begin{theorem}
\label{thm:divheredfreerestrictionreflections}
Let $W$ be a finite, complex 
reflection group with reflection arrangement 
$\CA(W)$ and let $V \ne X \in L(\CA(W))$. 
The restricted arrangement $\CA(W)^X$ is divisionally free 
if and only if it is hereditarily divisionally free.
\end{theorem}

\begin{proof}
The forward implication follows from 
Corollary \ref{cor:product-hereddivfree}, 
\cite[Prop.\ 2.11, Prop.\ 2.14]{orliksolomon:unitaryreflectiongroups}, 
and Theorems \ref{thm:divheredfreereflections}, 
\ref{thm:divfreerestrictions}, and \ref{thm:indfree}(b).
The reverse implication is obvious.
\end{proof}


\smallskip 
{\bf Acknowledgments}: 
We acknowledge 
support from the DFG-priority program 
SPP1489 ``Algorithmic and Experimental Methods in
Algebra, Geometry, and Number Theory''.


\bigskip

\bibliographystyle{amsalpha}

\newcommand{\etalchar}[1]{$^{#1}$}
\providecommand{\bysame}{\leavevmode\hbox to3em{\hrulefill}\thinspace}
\providecommand{\MR}{\relax\ifhmode\unskip\space\fi MR }
\providecommand{\MRhref}[2]{%
  \href{http://www.ams.org/mathscinet-getitem?mr=#1}{#2} }
\providecommand{\href}[2]{#2}


\end{document}